\documentclass{amsart}
\usepackage[utf8]{inputenc}
\usepackage{amssymb}
\usepackage{hyperref}
\usepackage[all]{xy}
\usepackage[final]{showkeys} 

\theoremstyle{definition}
\newtheorem{mydef}{Definition}[section]
\newtheorem{lem}[mydef]{Lemma}

\newtheorem{thm}[mydef]{Theorem}

\newtheorem{cor}[mydef]{Corollary}

\newtheorem{prop}[mydef]{Proposition}

\newtheorem{defin}[mydef]{Definition}
\newtheorem{example}[mydef]{Example}

\newtheorem{remark}[mydef]{Remark}
\newtheorem{rem}[mydef]{Remark}
\newtheorem{notation}[mydef]{Notation}
\newtheorem{fact}[mydef]{Fact}

\newcommand{\fct}[2]{{}^{#1}#2}

\newcommand{\pullbackcorner}[1][dl]{\save*!/#1-1pc/#1:(1,-1)@^{|-}\restore}

\newcommand\Po{\operatorname{Po}}
\newcommand\Tc{\operatorname{Tc}}
\newcommand\Rt{\operatorname{Rt}}

\newcommand\Iso{\operatorname{Iso}}

\newcommand\id{\operatorname{id}}

\newcommand{\Rmod}{R\text{-}\operatorname{\bf {Mod}}}

\newcommand\cof{\operatorname{cof}}

\newcommand\cell{\operatorname{cell}}
\newcommand\colim{\operatorname{colim}}

\newcommand\cb{\mathcal {B}}
\newcommand\cc{\mathcal {C}}

\newcommand{\Cc}{\mathcal{C}}

\newcommand\ck{\mathcal {K}}
\newcommand\cl{\mathcal {L}}
\newcommand\cm{\mathcal {M}}

\newcommand\cx{\mathcal {X}}

\newcommand{\coker}{\operatorname{coker}}

\newcommand{\cf}[1]{\text{cf} (#1)}
\newcommand{\seq}[1]{\langle #1 \rangle}

\newbox\noforkbox \newdimen\forklinewidth
\forklinewidth=0.3pt \setbox0\hbox{$\textstyle\smile$}
\setbox1\hbox to \wd0{\hfil\vrule width \forklinewidth depth-2pt
 height 10pt \hfil}
\wd1=0 cm \setbox\noforkbox\hbox{\lower 2pt\box1\lower
2pt\box0\relax}
\def\unionstick{\mathop{\copy\noforkbox}\limits}
\newcommand{\nf}{\unionstick}

\newcommand{\smallnf}{\downarrow}

\newcommand{\nfs}[4]{#2 \nf_{#1}^{#4} #3}

\def\1nf{\unionstick^{(1)}}

\def\2nf{\unionstick^{(2)}}
\def\3nf{\unionstick^{(3)}}

\newcommand{\NF}{\operatorname{NF}}

\newcommand{\Mm}{\mathcal{M}}

\newcommand{\Xx}{\mathcal{X}}

\title{Cellular categories and stable independence}
\date{\today }

\author[Lieberman]{Michael Lieberman}
\email{lieberman@math.muni.cz}
\urladdr{http://www.math.muni.cz/\textasciitilde lieberman/}
\address{Institute of Mathematics, Faculty of Mechanical Engineering, Brno University of Technology, Brno, Czech Republic}

\author[Rosick\'y]{Ji\v r\'i Rosick\'y}
\email{rosicky@math.muni.cz}
\urladdr{http://www.math.muni.cz/\textasciitilde rosicky/}
\address{Department of Mathematics and Statistics, Faculty of Science, Masaryk University, Brno, Czech Republic}
\thanks{The second author is supported by the Grant agency of the Czech Republic under the grant 19-00902S}

\author[Vasey]{Sebastien Vasey}

\begin{document}

\begin{abstract}
  We exhibit a bridge between the theory of \emph{cellular categories}, used in algebraic topology and homological algebra, and the model-theoretic notion of \emph{stable independence}. Roughly speaking, we show that the combinatorial cellular categories (those where, in a precise sense, the cellular morphisms are generated by a set) are exactly those that give rise to stable independence notions. We give two applications: on the one hand, we show that the abstract elementary classes of roots of Ext studied by Baldwin-Eklof-Trlifaj are stable and tame. On the other hand, we give a simpler proof (in a special case) that combinatorial categories are closed under 2-limits, a theorem of Makkai and Rosick\' y.
\end{abstract}

\maketitle


\section{Introduction}

Stable (nonforking) independence is a central notion of model theory. In the first-order context, it was introduced by Shelah in \cite{shelahfobook}, and constitutes an essential tool both in that book and in decades of subsequent work in first-order model theory.  In the now-dominant anchor notation introduced by Makkai, \cite{makkai-survey}, this is rendered as a relation on quadruples of sets, $\nfs{A}{B}{C}{D}$, understood to mean that $B$ and $C$ are independent---in a precise syntactic sense---over $A$ in $D$.  Generalizing (and serving many of the same purposes as) linear or algebraic independence, this independence notion can also be directly axiomatized as the canonical quaternary relation satisfying a list of essential properties, including, but not limited to invariance, monotonicity, uniqueness, and local character.  This notion was subsequently extended---again in syntactic form---to abstract elementary classes in
\cite{shclassaecs}, with corresponding direct axiomatization presented in \cite{bgkv-apal}.  In earlier work of the authors, \cite{indep-categ-advances}, it was shown that the latter axiomatization, in particular, leads naturally to the formulation of nonforking in an abstract category as a calculus of special commutative squares,
$$  
    \xymatrix@=3pc{
      C \ar[r]\ar@{}[dr]|\nf & D \\
      A \ar [u] \ar[r] & B \ar[u]
    }
    $$
identified as "independent."  That this specializes to stable independence in abstract elementary classes is proven in \cite[\S 8]{indep-categ-advances}; connections with the classical notion are examined in, e.g. \cite[5.7(6)]{vasinvite}.  It is shown in \cite{indep-categ-advances}, moreover, that this axiomatization is canonical in accessible categories with monomorphisms, assuming that they have chain bounds, thereby extending the canonicity theorem for abstract elementary classes of \cite{bgkv-apal}. The present paper hinges on a pair of observations implicit in the earlier one.  First, the aforementioned commutative squares can be thought as a replacement for pushouts in situations where they---an often indispensable tool in category-theoretic constructions---are not available: for example, when all morphisms are monomorphisms, which is typical for abstract elementary classes.  Second, we need not, in fact, assume that all morphisms are monomorphisms: we conclude in an appendix, for example, that canonicity holds in still greater generality, in arbitrary accessible categories with chain bounds.  More broadly, this shifts stable independence away from the model-theoretic framework, and means that the benefits of the first observation are applicable across the broad swathe of category theory.  The chief aim of the present paper is to highlight one particularly fruitful connection that arises as a result, between stable independence and the important homotopy-theoretic concepts of cellular and cofibrant generation.

Cellular categories were introduced in \cite{mr} as cocomplete categories equipped with a class of morphisms (called cellular)
containing all isomorphisms and closed under pushouts and transfinite compositions. These categories are abundant in homotopy theory
because any Quillen model category carries two cellular structures given by cofibrations and trivial cofibrations respectively. These cellular
categories are, in addition, retract-closed (in the category of morphisms). A retract-closed cellular category is cofibrantly
generated if it is generated by a set of morphisms using pushouts, transfinite compositions and retracts. In locally presentable categories, this implies that cellular morphisms form a left part of a weak factorization system. In \cite{mr}, retract-closed cofibrantly generated cellular locally presentable categories were called \emph{combinatorial}. The main result of \cite{mr} is that combinatorial
categories are closed under 2-limits, in particular under pseudopullbacks. A consequence is that combinatorial categories are left-induced in a sense that, given a colimit preserving functor $F:\ck\to\cl$ from a locally presentable category $\ck$ to a combinatorial
category $\cl$ then preimages of cellular morphisms form a combinatorial structure on $\ck$. This was later used, e.g., in \cite{induced-model-jtop}. The proof is quite delicate and depends on Lurie's concept of a good colimit (see \cite{fat-small-obj}).

The main result of the present paper is that, in the special case when cellular morphisms are coherent and $\aleph_0$-continuous, a retract-closed cellular category is combinatorial if and only if it carries a stable independence notion (Theorem \ref{cofib-gen}).     Independent squares coincide with \emph{cellular} squares; that is, squares of cellular morphisms such that the unique morphism from the pushout is cellular. These squares are also used in \cite{henry-induced}. Since a pre-image of an accessible category is accessible, this yields a simple proof that coherent and $\aleph_0$-continuous combinatorial categories are left-induced (see Corollary \ref{lazy}). While coherence is quite common, especially for trivial cofibrations, $\aleph_0$-continuity is more limiting. Nevertheless, our
theorem covers many situations. In particular, we will show (Theorem \ref{ext-thm}) that the abstract elementary classes of ``roots of Ext'' studied in \cite{bet} (for example the AEC of flat modules with flat monomorphisms) have a stable independence notion. Note, too, that since pure monomorphisms in a locally finitely presentable category are coherent and $\aleph_0$-continuous, the result of \cite{lprv-purecofgen-v3}, the proof of which relies on \cite{mr}, actually falls within the framework of this paper.

In an appendix, we prove a strengthening of the canonicity result of \cite{indep-categ-advances}: if a category $\ck$ has chain bounds, it has at most one weakly stable independence relation (in fact, we prove something stronger still, cf. Theorem~\ref{canon-thm}).  This new canonicity theorem eliminates the requirement---present in \cite{indep-categ-advances}---that all morphisms in $\ck$ are monomorphisms, and is central to several of the results of this paper.  Both Theorem~\ref{canon-thm} and its proof should be of independent interest: in connection with the latter, we show that it is possible to define, and to work with, independent sequences in an abstract category; that is, without reference to elements.

Concerning terminology, we will refer freely to \cite{adamek-rosicky}, \cite{mr} and \cite{indep-categ-advances} (concerning accessible
categories, cellular categories, and stable independence respectively). A more comprehensive version of the present paper, with added background, can be found at \url{https://arxiv.org/abs/1904.05691v2}.

\subsection{Acknowledgments}

We thank Jan Trlifaj for helpful conversations about roots of Ext and Simon Henry for sharing \cite{henry-induced} with us. We are also indebted to John Baldwin, Marcos Mazari-Armida, Misha Gavrilovich, and the anonymous referee for useful feedback. 

\section{Cellular categories}
Recall that a cocomplete category $\ck$ is called \textit{cellular} if it is equipped with a class $\cm$ of morphisms containing all isomorphisms and closed under pushouts and transfinite compositions (see \cite{mr}).

\begin{rem}\label{subcat}
A composition of two morphisms is a special case of a transfinite composition. Thus a cellular category $(\ck,\cm)$ induces
a subcategory $\ck_\cm$ of the category $\ck$ whose objects are those in $\ck$ and whose morphisms are precisely those of $\cm$.
Since $\cm$ contains all isomorphisms, the subcategory $\ck_\cm$ is \textit{isomorphism-closed}. Still, $\ck_\cm$ need not have pushouts.

In order to explain this, recall that $\cm$ is closed under pushouts whenever, given a pushout square
  $$  
    \xymatrix@=3pc{
      C \ar[r]^h & P \\
      A \ar [u]^g \ar[r]_f & B \ar[u]
    }
    $$
in $\ck$ with $f\in\cm$, then $h\in\cm$. But this does not mean that, if also $g\in\cm$, that this square is a pushout square
in $\ck_\cm$. The latter means that given another commutative square in $\ck_\cm$, as below, with $u,v\in \cm$,
   $$  
  \xymatrix@=3pc{
    & & D \\    
    B \ar[r]\ar@/^/[rru]^{u} & P \ar[ru]_t \pullbackcorner & \\
    A \ar [u]^g \ar[r]_f & C \ar[u]\ar@/_/[ruu]_{v} &
  }
  $$
then the induced morphism $t$ is in $\cm$. 

Similarly, although $\cm$ is closed under transfinite compositions, these composition does not to be colimits in $\ck_\cm$.
In the latter case, $\ck_\cm$ would be closed under colimits of smooth chains, which implies closure under all directed
colimits (see \cite[1.7]{adamek-rosicky}).
\end{rem}

\begin{defin}\label{cellular-square}
Let $(\ck,\cm)$ be a cellular category. A commutative square
$$  
    \xymatrix@=3pc{
      C \ar[r]^u & D \\
      A \ar [u]^g \ar[r]_f & B \ar[u]_v
    }
    $$
is called \emph{cellular} if the induced morphism $t:P\to D$ from the pushout (see above) belongs to $\cm$   .
\end{defin}

\begin{rem}
Cellular squares could also be called $\cm$-effective. In the special case in which $\cm$ is the class of regular monomorphisms, this corresponds precisely to the effective squares considered in \cite{indep-categ-advances}, and originating in \cite{effective-unions}.
\end{rem}

\begin{defin}\label{coherent-def}
A cellular category $(\ck,\cm)$ will be called
\begin{enumerate}
\item \emph{coherent} if whenever $f$ and $g$ are composable morphisms, $gf \in \Mm$ and $g \in \Mm$, then $f \in \Mm$,
\item \emph{left cancellable} if $gf \in \Mm$ implies $f \in \Mm$,
\item \emph{$\lambda$-continuous} if $\ck_{\Mm}$ is closed under $\lambda$-directed colimits in $\ck$,
\item \emph{$\lambda$-accessible} it is $\lambda$-continuous and both $\ck$ and $\ck_{\Mm}$ are $\lambda$-accessible.
\item \emph{accessible} if it is $\lambda$-accessible for some $\lambda$.
\end{enumerate}
\end{defin}

\begin{rem}\label{coherent-prop} \
  \begin{enumerate}
  \item Since a cellular category is cocomplete, an accessible cellular category has $\ck$ locally presentable.
  \item\label{coherent-prop-2} It is easy to see that $(\ck,\cm)$ is $\lambda$-continuous provided that $\cm$ is closed under $\lambda$-directed colimits 
in $\ck^2$. In fact, given a $\lambda$-directed diagram $D:I\to \ck_\cm$ and its colimit $\delta_i:Di\to K$ in $\ck$, then
$\delta_i=\colim_{i\leq j\in I}D_{i,j}$, where the $D_{i,j}:Di\to Dj$ are the appropriate diagram maps. Similarly, given a cocone $\gamma_i:Di\to L$
in $\ck_\cm$ then the induced morphism $g:K\to L$ is precisely $\colim_i\gamma_i$.
\end{enumerate}
\end{rem}

\begin{rem}\label{independence} \
  \begin{enumerate}
  \item In \cite{indep-categ-advances}, we defined an \emph{independence relation} (or \emph{independence notion}) in a category $\ck$ 
as a class $\nf$ of commutative square (called \emph{$\nf$-independent}, or just \emph{independent}, squares) such that, for any commutative diagram

    $$
    
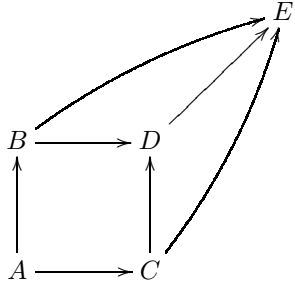
\xymatrix@=3pc{
     & & E \\
    B \ar[r]\ar@/^/[rru] & D \ar[ru] & \\
    A \ar [u] \ar[r] & C \ar[u]\ar@/_/[ruu] &
    }
    $$
    
{\noindent}the square spanning $A, B, C,$ and $D$ is independent if and only if the square spanning $A, B, C,$ and $E$ is independent.
A subcategory of $\ck^2$ (the category of morphisms in $\ck$, whose objects are morphisms and morphisms are commutative squares)
whose objects are morphisms and morphisms are independent squares was denoted as $\ck_{\NF}$. Here, we will denote it as 
$\ck_\downarrow$.

\item In \cite{indep-categ-advances}, as independence relation $\nf$ was defined to be \emph{stable} if it is symmetric, transitive, accessible, has existence, and has uniqueness.
In case $\nf$ satisfies all of the above conditions except accessibility, we say that it is \emph{weakly stable}. 

\item Accessibility of $\nf$ means that the category 
$\ck_\downarrow$ is accessible, which implies, in particular, that it is closed in $\ck^2$ under $\lambda$-directed colimits for some $\lambda$
(see \cite[3.26]{indep-categ-advances}). If $\nf$ satisfies the latter closure condition, we say that it is $\lambda$-\emph{continuous}.

\item\label{independence-3}  Accessibility of $\nf$ also implies that $\ck$ is accessible (see \cite[3.27]{indep-categ-advances}).

  \end{enumerate}
\end{rem}


\begin{thm}\label{weak-stable-thm}
  If $(\ck,\Mm)$ is a cellular category, then cellular squares form a weakly stable independence relation in $\ck_{\Mm}$.
\end{thm}
\begin{proof}
  We first check that cellular squares form an independence notion. Assume that $(A, B, C, D)$ is a commutative square\footnote{We occasionally economize by not explicitly naming the morphisms involved when there is no danger of confusion.} in $\ck_{\Mm}$ and we are given a morphism $D \to E$ in $\Mm$. If $(A, B, C, D)$ is cellular, then closure of $\Mm$ under composition yields that 
  $(A, B, C, E)$ is cellular. Conversely, if $(A, B, C, E)$ is cellular, then the map $P \to E$ from the pushout is in $\Mm$ by assumption, and also $D \to E$ is in $\Mm$, so by coherence also the map $P \to D$ is in $\Mm$. Thus $(A, B, C, D)$ is cellular.

  This concludes the proof that cellular squares form an independence notion. Of course, the relation is also symmetric. Existence follows from closure under pushouts (and the fact that the identity map is an isomorphism, hence in $\Mm$). In order to prove the uniqueness property, consider cellular squares $(A, B, C, D^1)$ and $(A,B,C,D^2)$ with the same span $B \leftarrow A\to C$. Form the pushout 
$$
\xymatrix@=3pc{
B \ar[r]^{} & P \\
A \ar [u]^{} \ar [r]_{} &
C \ar[u]_{}
}
$$
and take the induced morphisms $P\to D^1$ and $P\to D^2$. They are in $\Mm$ by cellularity. Then the pushout

$$
\xymatrix@=3pc{
D^1 \ar[r]^{} & D \\
P \ar [u]^{} \ar [r]_{} &
D^2 \ar[u]_{}
}
$$

amalgamates the starting diagram.

To prove transitivity, consider:
  
$$  
  \xymatrix@=3pc{
    B \ar[r]^{} & D \ar[r]^{}  & F \\
    A \ar [u]^{f} \ar [r]^{g} & C \ar[u]_{} \ar[r]^{g'} & E \ar[u]_{}
  }
  $$
where both squares are cellular. We have to show that the outer rectangle is cellular. Thus we have to show that the induced morphism $p:P\to F$ from the pushout
$$
\xymatrix@=3pc{
B \ar[r]^{} & P \pullbackcorner\\
A \ar [u]^{f} \ar [r]^{g'g} &
E \ar[u]_{}
}
$$  
is in $\Mm$. This pushout is a composition of pushouts
$$  
  \xymatrix@=3pc{
    B \ar[r]^{u} & Q \ar[r]^{u'}\pullbackcorner  & P \pullbackcorner\\
    A \ar [u]^{f} \ar [r]^{g} & C \ar[u]_{v} \ar[r]^{g'} & E \ar[u]_{v'}
  }
  $$
Recalling the left square of the starting diagram, we have an induced morphism $q:Q\to D$.  Consider the pushout
 $$
\xymatrix@=3pc{
D \ar[r]^{} & P' \\
Q \ar [u]^{q} \ar [r]^{u'} &
P \ar[u]_{\bar{q}}
}
$$  
Since the left square of the starting diagram is cellular, $q$ is in $\Mm$ and thus $\bar{q}$ is in $\Mm$. Composing this pushout with the right pushout square in the diagram above it, we obtain the pushout
$$
\xymatrix@=3pc{
D \ar[r]^{} & P' \\
C \ar [u]^{} \ar [r]_{} &
E \ar[u]_{}
}
$$ 
The right square in the starting diagram is cellular, so the induced morphism $p':P'\to F$ is in $\Mm$. Thus $p=p'\bar{q}$ is in $\Mm$.
\end{proof}
  
\begin{remark}
In the proof, we have not used the full strength of the assumption that $\cm$ is closed under transfinite compositions: here finite compositions suffice.
Coherence is used only once, in the proof that cellular squares form an independence notion (specifically, in the proof that the top right corner can be made ``smaller''). Instead of coherence, we could also have assumed the dual property, cocoherence: indeed, we know in the proof that the maps $C \to D$ and $C \to P$ are in $\Mm$, so cocoherence would give us immediately that $P \to D$ is in $\Mm$. Note, however, that if $\Mm$ is a class of monomorphisms, cocoherence is too strong an assumption: if a section $i: A \to B$ is in $\Mm$, cocoherence would imply that the corresponding retract $r: B \to A$ is in $\Mm$, and so $r$ would have to be an isomorphism.
\end{remark}

\begin{notation}
For a cellular category $(\ck,\Mm)$, we write $\ck_{\Mm,\smallnf}$ for $\left(\ck_{\Mm}\right)_{\smallnf}$.
\end{notation}

\begin{remark}
In a cellular category, cellular squares form a class of morphisms in $\ck^2$. Following Theorem \ref{weak-stable-thm} this class
is closed under composition, by transitivity of the associated weakly stable independence notion. Using \cite[3.18]{indep-categ-advances}, it is isomorphism-closed. Using
\cite[3.20, 3.21]{indep-categ-advances}, cellular squares are left-cancellable.  
\end{remark}

\begin{lem}\label{dir-colim-lem}
If $(\ck,\Mm)$ is a $\lambda$-continuous cellular category, then the independence relation given by cellular squares is $\lambda$-continuous.
\end{lem}
\begin{proof}
Let $(\ck,\Mm)$ be $\lambda$-continuous. Let $D:I\to \ck_{\Mm, \smallnf}$ be a $\lambda$-directed diagram where $Di$ is $f_i:A_i\to B_i$. Let $f:A\to B$ be a colimit of $D$ in $\left(\ck_{\Mm}\right)^2$. For each $i\in I$, the pushout of the colimit coprojection $A_i\to A$ along $f_i$, i.e.
$$
\xymatrix@=3pc{
A \ar[r]^{g} & P \\
A_i \ar [u]^{} \ar [r]_{f_i} &
B_i \ar[u]_{}
}
$$
is a $\lambda$-directed colimit of pushouts
$$
\xymatrix@=3pc{
A_{i'} \ar[r]^{g_{i'}} & P_{i'} \\
A_i \ar [u]^{} \ar [r]_{f_i} &
B_i \ar[u]_{}
}
$$
Thus the induced morphism $p:P\to B$ is a $\lambda$-directed colimit of induced morphisms $p_{i'}:P_{i'}\to B_{i'}$. Since $\Mm$ is 
$\lambda$-continuous, it follows that $p \in \Mm$. This shows that all the maps of the cocone $(f_i \to g)_{i \in I}$ are independent squares. Similarly, one can check that this is a colimit cocone in $\ck_{\Mm, \smallnf}$. Thus $\ck_{\Mm, \smallnf}$ is closed under 
$\lambda$-directed colimits in $\left(\ck_{\Mm}\right)^2$.  
\end{proof}

\section{Combinatorial categories}
A cellular category $(\ck,\cm)$ is said to be \emph{retract-closed} if $\cm$ is closed under retracts in the category $\ck^2$. A retract-closed
cellular category is called \emph{combinatorial} if it is \emph{cofibrantly generated}, i.e., if $\cm$ is the closure of a set $\cx$ of morphisms under pushouts, transfinite compositions and retracts. In particular, $\cm=\cof(\cx)$, where
$$
\cof(\cx)=\Rt(\Tc(\Po(\cx)))=\Rt(\cell(\cx))
$$
where $\Po$ denotes the closure under pushouts, $\Tc$ under transfinite compositions and $\Rt$ under retracts (see \cite{mr}).

For $\lambda$ a regular cardinal, we write $\ck_\lambda$ for the full subcategory of $\ck$ consisting of $\lambda$-presentable objects. We similarly denote by $\ck_\lambda^2$ the full subcategory of $\ck^2$ consisting of morphisms with $\lambda$-presentable domains and codomains. We will also write, for example, $\Mm_\lambda := \Mm \cap \ck_\lambda^2$.

The next result, the main theorem of this paper, characterizes when cellular squares form a stable independence notion in terms of cofibrant generation of the corresponding class of morphisms.  

To go from stable independence to cofibrant generation, we require a technical result from \cite[\S9]{internal-improved-v2} concerning the existence of \emph{filtrations}. Recall that the \emph{presentability rank} of an object $A$ is the least regular cardinal $\lambda$ such that $A$ is $\lambda$-presentable. We say that $A$ is \emph{filtrable} if it can be written as the directed colimit of a chain of objects with lower presentability rank than $A$. We say that $A$ is \emph{almost filtrable} if it is a retract of such a chain. The chain is \emph{smooth} if directed colimits are taken at every limit ordinal. By \cite[9.12]{internal-improved-v2}, in any accessible category with directed colimits, there exists a regular cardinal $\lambda$ such that any object with presentability rank at least $\lambda$ is almost filtrable (and, moreover, the chain in the filtration can be chosen to be smooth). We say that a category satisfying the latter condition is \emph{almost well $\lambda$-filtrable}.

\begin{thm}[Main theorem]\label{cofib-gen}
Let $(\ck,\cm)$ be an accessible cellular category which is retract-closed, coherent and $\aleph_0$-continuous. The following are equivalent:

  \begin{enumerate}
  \item\label{cofib-gen-1} $\ck_{\Mm}$ has a stable independence notion.
  \item\label{cofib-gen-2} Cellular squares form a stable independence notion in $\ck_{\Mm}$.
  \item\label{cofib-gen-3} $(\ck,\Mm)$ is combinatorial.
  \end{enumerate}
\end{thm}
\begin{proof}  
(\ref{cofib-gen-1}) implies (\ref{cofib-gen-2}): If $\ck_{\Mm}$ has a stable independence notion, then canonicity (Theorem \ref{canon-thm} -- note that $\ck_{\Mm}$ has directed colimits, since $\Mm$ is $\aleph_0$-continuous) together with Theorem \ref{weak-stable-thm} ensures that it is given by cellular squares. Note that if we know that all morphisms in $\Mm$ are monos, then we do not need Theorem \ref{canon-thm} and can use \cite[9.1]{indep-categ-advances} instead.
  
(\ref{cofib-gen-2}) implies (\ref{cofib-gen-3}): Assume that $\ck_{\Mm}$ has a stable independence $\nf$ given by cellular squares. Thus $\ck_{\Mm, \smallnf}$ is accessible and has directed colimits (by Lemma \ref{dir-colim-lem}). By Remark \ref{independence}(\ref{independence-3}), 
$\ck_{\Mm}$ is accessible, so $\Mm$ is accessible. Using the preceding discussion, pick a regular uncountable cardinal $\lambda$ such both $\ck$ and $\ck_{\Mm, \smallnf}$ are $\lambda$-accessible and almost well $\lambda$-filtrable. Let $\Mm_\lambda$ be the collection of morphisms in $\Mm$ whose domains and codomains are $\lambda$-presentable (in $\ck$). We will show that for each infinite cardinal 
$\mu$, $\Mm_{\mu^+} \subseteq \cof (\Mm_\lambda)$. We proceed by induction on $\mu$. When 
$\mu < \lambda$, this is trivial, so assume that $\mu \ge \lambda$. Note that, playing with pushouts, it is straightforward to check that the $\mu^+$-presentable objects in $\ck_{\Mm, \smallnf}$ are exactly the morphisms of $\Mm_{\mu^+}$. 

Every morphism $h$ in $\Mm_{\mu^+}$ must be a retract of a filtrable object in $\ck_{\Mm, \smallnf}$. Now, retracts in $\ck_{\Mm, \smallnf}$ are retracts in $\ck^2$, so since we are looking at $\cof (\Mm_\lambda)$ it suffices to show that any morphism $h$ in $\Mm_{\mu^+}$ which \emph{is} filtrable in $\ck_{\Mm, \smallnf}$ is in $\cof (\Mm_\lambda)$. So take such a morphism. Write $h = h_0:K_0\to L$. We will show that  $h_0\in\cof (\Mm_\lambda)$. Express $h_0$ as a colimit of a smooth chain of morphisms $t_{0i}\in\cof (\Mm_\lambda)$, $i<\cf{\mu}$, between $(<\mu^+)$-presentable objects in $\ck_{\Mm, \smallnf}$.
$$  
      \xymatrix@=3pc{
        K_0 \ar@{}\ar[r]^{h_0} & L \\
        K_{0i} \ar [u]^{k_{0i}} \ar [r]_{t_{0i}} &
        L_{0i} \ar[u]_{l_{0i}}
      }
      $$
 
Form a pushout
$$  
      \xymatrix@=3pc{
        K_0 \ar@{}\ar[r]^{h_{01}} & K_1 \\
        K_{00} \ar [u]^{k_{00}} \ar [r]_{t_{00}} &
        L_{00} \ar[u]_{\bar{k}_{00}}
      }
      $$
and take the induced morphism $h_1:K_1\to L$. Since the starting square is cellular, $h_1$ is in $\Mm$. Note also that $K_1$ is $\mu^+$-presentable. We have a commutative square
$$  
      \xymatrix@=3pc{
        K_1 \ar@{}\ar[r]^{h_1} & L \\
        K_{01} \ar [u]^{h_{01}k_{01}} \ar [r]_{t_{01}} &
        L_{01} \ar[u]_{l_{01}}
      }
      $$
because $h_1 h_{01} k_{01} = h_0 k_{01} = l_{01} t_{01}$. We can express $h_1$ as a colimit of a smooth chain of morphisms $t_{1i}\in\cof (\Mm_\lambda)$, $1\leq i<\cf{\mu}$, between $<\mu^+$-presentable objects in $\ck_{\Mm, \smallnf}$ which are above $t_{01}$
      
$$
\xymatrix@C=3pc@R=3pc{
K_1 \ar [r]^{h_1}  & L \\
K_{1i} \ar[r]^{t_{1i}} \ar [u]^{k_{1i}}  & L_{1i} \ar [u]_{l_{1i}}\\
K_{01} \ar [r]_{t_{01}} \ar [u]^{} & L_{01} \ar [u]_{}
}
$$
 
Form a pushout
$$  
      \xymatrix@=3pc{
        K_1 \ar@{}\ar[r]^{h_{12}} & K_2 \\
        K_{11} \ar [u]^{k_{11}} \ar [r]_{t_{11}} &
        L_{11} \ar[u]_{\bar{k}_{11}}
      }
      $$
and take the induced morphisms $h_2:K_2\to L$. Again, by cellularity, $h_2$ is in $\Mm$. In
$$
K_0 \xrightarrow{\ h_{01}\ } K_1\xrightarrow{\ h_{12}\ } K_2\xrightarrow{\ h_2} L
$$
we put $h_{02}=h_{12}h_{01}$ and continue transfinitely. This means that for $i<\cf{\mu}$ we express
$h_i$ as a colimit of a smooth chain of morphisms $t_{ij}\in\cof (\Mm_\lambda)$, $i\leq j<\cf{\mu}$, between $(<\mu^+)$-presentable objects in $\ck_{\Mm, \smallnf}$ which are above $t_{0i}$

$$
\xymatrix@C=3pc@R=3pc{
K_i \ar [r]^{h_i}  & L \\
K_{i+1,j} \ar[r]^{t_{i+1,j}} \ar [u]^{k_{i+1,j}}  & L_{i+1,j} \ar [u]_{l_{i+1,j}}\\
K_{0i} \ar [r]_{t_{0i}} \ar [u]^{} & L_{0i} \ar [u]_{}
}
$$

Form a pushout
$$  
      \xymatrix@=3pc{
        K_i \ar@{}\ar[r]^{h_{i,i+1}} & K_{i+1} \\
        K_{i+1,i+1} \ar [u]^{k_{i+1,i+1}} \ar [r]_{t_{i+1,i+1}} &
        L_{i,i+1} \ar[u]_{\bar{k}_{i+1,i+1}}
      }
      $$
and take the induced morphisms $h_{i+1}:K_{i+1}\to L$. By cellularity, $h_{i+1}$ is in $\Mm$. We put $h_{k,i+1}=h_{i,i+1}h_{ik}$. At limit steps we take colimits. Then by construction $L=K_{\cf{\mu}}$ and $h_0$ is the transfinite composition of $(h_{ij})_{i<j<\cf{\mu}}$. We have just observed that each $h_{ij}$ is in $\cof (\Mm_\lambda)$, so $h_0$ also is.

(\ref{cofib-gen-3}) implies (\ref{cofib-gen-1}): Assume that $\Mm$ is accessible and cofibrantly generated in $\ck$. Let $\Xx$ be a subset of $\Mm$ so that $\Mm = \cof (\Xx)$. Let $\lambda$ be a big-enough uncountable regular cardinal such that $\ck$ and $\ck_{\Mm}$ are $\lambda$-accessible, and all the morphisms in $\Xx$ have $\lambda$-presentable domain and codomain. Note that, by coherence, for any regular $\mu \ge \lambda$, an object which is $\mu$-presentable in $\ck$ is $\mu$-presentable in $\ck_{\Mm}$. We claim that $\ck_{\Mm, \smallnf}$ is $\lambda$-accessible. First, $\ck_{\Mm, \smallnf}$ is closed under directed colimits in $\ck_{\Mm}$ by Lemma \ref{dir-colim-lem}. Now let $\Mm_\lambda$ be the class of morphisms in $\Mm$ with $\lambda$-presentable domain and codomain and let $\Mm^\ast$ be the class of morphisms in $\Mm$ that are $\lambda$-directed colimit (in $\ck_{\Mm, \smallnf}$) of morphisms in $\Mm_\lambda$. It suffices to see that $\Mm^\ast = \Mm$.

First, any pushout of a morphism in $\Mm_\lambda$ is in $\Mm^\ast$. Consider such a pushout
$$  
      \xymatrix@=3pc{
        K \ar@{}\ar[r]^h{} & L \\
        K_0 \ar [u]^{k_0} \ar [r]_{h_0} &
        L_0 \ar[u]_{l_0}
      }
      $$
where $K_0$ and $L_0$ are $\lambda$-presentable. Then $K$ is a $\lambda$-directed colimits of $\lambda$-presentable objects $K_i$ above $K_0$ in $\ck_{\Mm}$.
Consider pushouts
$$
\xymatrix@C=3pc@R=3pc{
K \ar [r]^{h}  & L \\
K_i \ar[r]^{h_i} \ar [u]^{}  & L_i \ar [u]_{}\\
K_0 \ar [r]_{h_0} \ar [u]^{} & L_0 \ar [u]_{}
}
$$ 
It is easy to check that the $L_i$'s are also $\lambda$-presentable and that $h=\colim h_i$ in $\ck_{\Mm, \smallnf}$. Thus $h \in \Mm^\ast$.

Second, $\Mm^\ast$ is closed under compositions of morphisms from $\Po_\lambda$ where $\Po_\lambda$ consists of pushouts of morphisms from $\Mm_\lambda$. Let $f:K\to L$ and $g:L\to M$ belong to $\Po_\lambda$. As above, $f$ is a $\lambda$-directed colimit (in $\ck_{\Mm, \smallnf}$), $(k_i,l_i):f_i\to f$ of $f_i\in\Mm_\lambda$, $f_i : K_i \to L_i$. Moreover, $g$ is a pushout of $g_0:L_0\to M_0$ having $L_0$ and $M_0$ both $\lambda$-presentable. Without loss of generality, we can assume that $L_0\to L$ factors through the $L_i$. We then take pushouts as above
$$
\xymatrix@C=3pc@R=3pc{
L \ar [r]^{g}  & M \\
L_i \ar[r]^{g_i} \ar [u]^{}  & M_i \ar [u]_{}\\
L_0 \ar [r]_{g_0} \ar [u]^{} & M_0 \ar [u]_{}
}
$$
This shows that $gf$ is a $\lambda$-directed colimit of the $g_if_i$'s in $\ck_{\Mm, \smallnf}$.

Third, $\Mm^\ast$ is closed under transfinite compositions of morphisms from $\Po_\lambda$. Let $(f_{ij})_{i,j\leq \alpha}$ be such a transfinite composition. At limit steps, $f_{0i}$ is the following directed colimit in $\ck_{\Mm, \smallnf}$:
$$  
      \xymatrix@=3pc{
        K_0 \ar@{}\ar[r]^{f_{0i}} & K_i \\
        K_0 \ar [u]^{\id} \ar [r]_{f_{0j}} &
        K_j \ar[u]_{f_{ji}}
      }
      $$

This shows that $f_{0,i}$ is in $\Mm^\ast$ (we used \cite{indep-categ-advances}, 3.12). 
      
We have shown that any transfinite composition of pushouts from $\Mm_\lambda$ is in $\Mm^\ast$. That is, $\cell (\Mm_\lambda) = \Tc (\Po (\Mm_\lambda)) \subseteq \Mm^\ast$. Since $\Mm$ is closed under pushouts, retracts, and transfinite compositions, $\cof (\Xx) \cap \ck_\lambda^2 \subseteq \Mm_\lambda$. By \cite{fat-small-obj}, B1, it follows that $\Mm = \cof (\Xx) = \cell (\Mm_\lambda)$. We deduce that $\Mm = \Mm^\ast$, as desired.
\end{proof}


\begin{example} \
  \begin{enumerate}
  \item On any locally presentable category $\ck$, there are two trivial cellular structures -- the \emph{discrete} $(\ck,\Iso)$ and the \emph{indiscrete} $(\ck,\ck^2)$. They are both combinatorial  (see \cite{mr}), coherent and $\aleph_0$-continuous. The first one is not accessible because $\ck_{\Iso}$ is not accessible (as long as $\ck$ is not small, in any case). The second is accessible and yields a stable independence relation where every commutative square is independent.
  \item On every locally presentable category $\ck$, there is a cellular structure where $\cm$ consists of regular monomorphisms. This cellular category is accessible, retract-closed and coherent. If $\ck$ is locally finitely presentable, it is $\aleph_0$-continuous. Concrete examples include graphs with induced subgraph embeddings, groups, Banach spaces, boolean algebras, Hilbert spaces, and any Grothendieck topos. The last two are combinatorial, hence have a stable independence notion. See \cite{indep-categ-advances} for more details.
  \item On every locally finitely presentable category $\ck$, there is a cellular structure where $\cm$ consists of pure monomorphisms.
This cellular category is accesible, retract-closed, coherent and $\aleph_0$-continuous. When this cellular structure is combinatorial is discussed in \cite{purity-algebra} and \cite{lprv-purecofgen-v3}. For example, the latter shows that $(\ck, \cm)$ is combinatorial for any additive category $\ck$.
  \end{enumerate}
\end{example}

Often, it is natural to look not at all objects, but just those objects $A$ so that $0 \to A$ is in $\Mm$ (where $0$ is an initial
object):

\begin{defin}\label{fib-def}
Let $(\ck,\cm)$ be a cellular category. An object $A$ is called \emph{cellular} if $0\to A$ is cellular. Let $\cc$ denote the full subcategory of $\ck$ consisting of cellular objects.
\end{defin}

\begin{rem}\label{fib-rmk}
Let $\cm_0$ be the class of cellular morphisms with a cellular domain (then the codomain is cellular too). Then $(\cc,\cm_0)$
satisfies all properties of a cellular category up to cocompleteness of $\cc$. Thus it induces a subcategory $\cc_{\cm_0}$ of $\cc$
consisting of cellular objects and cellular morphisms.

If $\cm$ is coherent, then every cellular morphism $A\to B$ with $B\in\cc$ has $A\in\cc$.
\end{rem}

We have the following version of Theorem \ref{cofib-gen} for cofibrant objects. Its advantage is that we do not need to assume that
$(\ck,\cm)$ itself is accessible: it suffices to have $\ck$ accessible.

\begin{thm}\label{cofib-gen-cof}
Let $\ck$ be a retract-closed, coherent and $\aleph_0$-continuous cellular category such that $\ck$ is accessible. The following are equivalent:

  \begin{enumerate}
  \item $\Cc_{\Mm_0}$ has a stable independence notion.
  \item $\Mm_0$-effective squares form a stable independence notion in $\Cc_{\Mm_0}$.
  \item $\Mm_0$ is cofibrantly generated in $\Cc$.
  \end{enumerate}
\end{thm}
\begin{proof}
Similar to the proof of Theorem \ref{cofib-gen}. Following \cite{fat-small-obj} 5.2, (3) implies that $\cc_{\cm_0}$ is accessible.
\end{proof}

In many cases, the cellular squares will be pullback squares:

\begin{fact}[\cite{ringel}, {\cite[11.15]{joy-of-cats}}]
Let $(\ck,\Mm)$ be a cellular category where every cellular morphism is a monomorphism. If:

  \begin{enumerate}
  \item A pullback of two morphisms in $\Mm$ is again in $\Mm$.
  \item Every epimorphism in $\Mm$ is an isomorphism.
  \end{enumerate}

  Then every cellular square is a pullback square.
\end{fact}

Conversely, it is natural to ask whether every pullback square is cellular. When $\Mm$ is the class of regular monomorphisms, categories with this property are said to have \emph{effective unions}, a condition isolated by Barr \cite{effective-unions}.  The connections of this special case with stable independence were investigated in \cite[\S5]{indep-categ-advances}, where it was shown that having effective unions implies that effective squares form a stable independence notion. We show that the definition can be naturally parameterized by $\Mm$ (this was done already for pure morphisms in \cite[2.2]{purity-algebra}), and the corresponding results generalized. 

\begin{defin}\label{effective-unions-def}
We say that a cellular category $(\ck,\cm)$ has \emph{effective unions} if

  \begin{enumerate}
  \item The pullback of any two morphisms in $\Mm$ with common codomain exists and the projections are again in $\Mm$.
  \item Any pullback square with morphisms in $\Mm$ is cellular.
  \end{enumerate}
\end{defin}

\begin{thm}\label{effective-unions-thm}
Let $(\ck,\cm)$ be a cellular category which is coherent, has effective unions, and with $\ck$ accessible. Then $(\ck, \Mm)$ is accessible if and only if cellular squares form a stable independence notion in $\ck_{\Mm}$.
\end{thm}
\begin{proof}
  If there is a stable independence notion in $\ck_{\Mm}$, then by Remark \ref{independence}(\ref{independence-3}), $(\ck, \Mm)$ is accessible. Let us prove the converse. Pick a regular cardinal $\lambda$ such that $(\ck, \Mm)$ is $\lambda$-accessible. By Theorem \ref{weak-stable-thm}, cellular squares form a weakly stable independence notion and by Lemma \ref{dir-colim-lem} this independence notion is $\lambda$-continuous. It remains to see that $\ck_{\Mm, \smallnf}$ is accessible. Consider an object $C \to D$ of $\ck_{\Mm,\smallnf}$. Since $\Mm$ is $\lambda$-accessible, $D$ can be written as a $\lambda$-directed colimit $\seq{D_i : i \in I}$ of $\lambda$-presentable objects. Let $C_i$ be the pullback of $C$ and $D_i$ over $D$. Then the resulting maps $C_i \to D_i$ form a $\lambda$-directed system. Since $\lambda$-directed colimits commute with finite limits (see \cite[1.59]{adamek-rosicky}, the pullback functor is accessible so must preserve arbitrarily large presentability ranks. Thus there is a bound on the presentability rank of $C_i$ that depends only on $\lambda$. This shows that $\ck_{\Mm, \smallnf}$ is accessible.  
\end{proof}

Note that, as opposed to Theorem \ref{cofib-gen}, we did \emph{not} need to assume that $(\ck, \Mm)$ was $\aleph_0$-continuous (nor that $(\ck, \Mm)$ was retract-closed). However, a category may fail to have effective unions even if the effective squares form a stable independence notion (this is the case, for example, in locally finite graphs with regular monos, see \cite[5.7]{indep-categ-advances}).

As a corollary, we obtain a quick proof that having effective unions implies cofibrant generation. This had been done ``by hand'' before for several special classes of morphisms \cite[1.12]{beke-sheafifiable}, \cite[2.4]{purity-algebra}.

\begin{cor}
  If $(\ck,\cm)$ is an accessible cellular category which is coherent, $\aleph_0$-continuous, and has effective unions, then it is combinatorial.
\end{cor}
\begin{proof}
  By Theorem \ref{effective-unions-thm} cellular squares form a stable independence notion, so by Theorem \ref{cofib-gen} (noting that retract-closedness is not used for this direction)  $(\ck, \Mm)$ is cofibrantly generated.
\end{proof}
 
\begin{rem}\label{left-induced}
Let $F:\ck\to\cl$ be a colimit-preserving functor from a locally presentable categopry $\ck$ to a combinatorial category $\cl$.
We get a cellular structure on $\ck$ where $f$ is cellular if and only if $Ff$ is cellular. This cellular structure is called
\emph{left-induced} (see \cite[3.8]{mr}). It was shown in \cite{mr}, using a great deal of heavy machinery, that such left-induced cellular structures are combinatorial. With the aid of Theorem~\ref{cofib-gen}, we obtain a special case of this result without any effort.
\end{rem}

\begin{cor}\label{lazy}
Let $F:\ck\to\cl$ be a colimit preserving functor from a locally presentable category to a combinatorial category.
If $\cl$ is coherent and $\aleph_0$-continuous, then $\ck$ is combinatorial.
\end{cor} 
\begin{proof}
Preimages of cellular squares are cellular and the left-induced cellular category $\ck$ is clearly retract-closed, coherent 
and  $\aleph_0$-continuous. We have a pseudopullback
$$  
    \xymatrix@=3pc{
      \ck^2 \ar[r]^{F^2} &\cl^2 \\
      \ck_\downarrow \ar [u] \ar[r]_{F_\downarrow} & \cl_\downarrow \ar[u]
    }
    $$
Since a pseudopullback of accessible categories is accessible (see \cite[Ex.\ 2n]{adamek-rosicky}), $\nf$ in $\ck$ is accessible. The now result follows from \ref{cofib-gen}.   
\end{proof}

\section{Abstract elementary classes of roots of Ext}

Abstract elementary classes (or AECs) are a framework for abstract model theory introduced by Shelah \cite{sh88}. We will use the category-theoretic characterization of Beke-Rosický \cite{beke-rosicky}: they are accessible categories with directed colimits and with all morphisms monomorphisms which embed ``nicely'' into finitely accessible categories.

\begin{lem}\label{aec}
  Let $(\ck,\cm)$ be an accessible cellular category which is coherent and $\aleph_0$-continuous. Assume that $\ck$ is finitely accessible and all morphisms in $\Mm$ are monomorphisms.

  \begin{enumerate}
  \item $\ck_{\cm}$ is an abstract elementary class.
  \item If $(\ck, \Mm)$ is combinatorial, then $\cc_{\cm_0}$ (see Definition \ref{fib-def}, Remark \ref{fib-rmk}) is an abstract elementary class.
  \end{enumerate}
\end{lem}
\begin{proof}
  It is easy to verify that $\ck_\cm$ satisfies the conditions in \cite[5.7]{beke-rosicky}. When $(\ck, \Mm)$ is combinatorial, one can use \cite[5.2]{fat-small-obj} to see that $\cc_{\cm_0}$ is an AEC as well.
\end{proof}

In what follows, we will apply our main theorem to the AECs studied in \cite{bet}. For a fixed (associative and unital) ring $R$, let $\Rmod$ denote the category of (left) $R$-modules with homomorphisms. It is a locally finitely presentable category.

\begin{defin} Given a class $\cb$ of $R$-modules, we define its \emph{Ext-orthogonality class}, $\fct{\perp_\infty}{\cb}$, as follows:

  $$\fct{\perp_\infty}{\cb}=\{A\,:\,\mbox{Ext}^i(A,N)=0\mbox{ for all } 1 \le i < \omega \text{ and all } N\in \cb\}$$
\end{defin}
  
Roughly speaking, $\fct{\perp_\infty}{\cb}$ is the collection of $R$-modules that do not admit nontrivial extensions by modules in $\cb$. For example, when $\cb$ is the class of all pure injective modules, then $\fct{\perp_\infty}{\cb}$ is exactly the class of flat modules (see \cite[5.3.22, 7.1.4]{enochs-jenda}).

From now on, we assume that $\cb$ is a class of pure injective modules.  Let $\ck := \Rmod$, and let $\Cc$ be the full subcategory of $\Rmod$ with objects from $\fct{\perp_\infty}{\cb}$. Let $\Mm$ be the class of monomorphisms (in $\Rmod$) whose cokernel is in $\fct{\perp_\infty}{\cb}$.  That is, a monomorphism $A \xrightarrow{f} B$ is in $\Mm$ if and only if $B / f[A]$ is in $\fct{\perp_\infty}{\cb}$. Let $\Mm_0$ be the class of elements in $\Mm$ with domain and codomain in $\Cc$. Note that this coincides with the notation from Definition \ref{fib-def}, Remark \ref{fib-rmk}.

The category $\Cc_{\Mm_0}$ is studied from the point of view of model theory by Baldwin-Eklof-Trlifaj \cite{bet}, where they prove it is an AEC. They ask (see \cite[4.1(1)]{bet}) what one can say about tameness and stability in $\Cc_{\Mm_0}$ (see, for example, \cite{baldwinbook09} for the relevant definitions). We now show, using our main theorem and known facts, that $\Cc_{\Mm_0}$ has a stable independence notion, hence (by \cite[8.16]{indep-categ-advances}) it will \emph{always} be stable and tame. 

\begin{thm}\label{ext-thm}
  $(\ck, \Mm)$ is a coherent, $\aleph_0$-continuous, and retract-closed cellular category. Moreover, $\Cc_{\Mm_0}$ is cofibrantly generated in $\Cc$. In particular, $\Cc_{\Mm_0}$ is an AEC with a stable independence notion.
\end{thm}
\begin{proof}
  The ``in particular'' part of the statement follows from Theorem \ref{cofib-gen-cof} and Lemma \ref{aec}. For the first sentence, following \cite[4.2]{flat-covers-factorizations}, $(\ck, \Mm)$ is a retract-closed cellular category. The coherence was observed in \cite[1.14]{bet} and $\aleph_0$-continuity in \cite[1.6]{bet}. In fact, the latter follows from \ref{coherent-prop}(\ref{coherent-prop-2}) because $\ck_{\Mm}$ is closed under directed colimits in $\Rmod$ (as outlined in, for example, \cite[\S1]{bet}). It remains to see that $\Cc_{\Mm_0}$ is cofibrantly generated in $\Cc$.

  By \cite[Proposition 2]{flat-cover}, \cite[Theorem 8]{covers-ext}, $\Cc_{\Mm_0}$ has refinements. This means there exists a regular cardinal $\theta$ so that any object of $\Cc$ can be written as the union of an increasing smooth chain $\seq{A_i : i < \alpha}$ of submodules, with $A_0$ the zero module and for all $i < \alpha$, $A_{i + 1} / A_i$ in $\Cc$ and $\theta$-presentable.

  By the proof of \cite[4.5]{flat-covers-factorizations}, $\Mm$ is cofibrantly generated by a set of maps $f$ so that $0 \to A \xrightarrow{f} F \to B \to 0$ is a short exact sequence, $F$ is a free module, and $B$ is a $\theta$-presentable object of $\cl$. Since $F$ is free, $F \in \Cc$ as well, hence $A \in \Cc$. Thus $f \in \Mm_0$. Thus $\Mm$ is cofibrantly generated in $\ck$ by a subset of $\Mm_0$, showing in particular that $\Mm_0$ is cofibrantly generated in $\Cc$.
\end{proof}

In this case, the cellular squares can be given a very concrete description:

\begin{prop}
A square
$$
\xymatrix@C=4pc@R=3pc{
A \ar [r]^{u} \ar [d]_{f}& C \ar [d]^g\\
B \ar [r]_{v}& D
}
$$
is cellular if and only if the pushout
$$
\xymatrix@C=4pc@R=3pc{
D \ar [r]^{\coker g} \ar [d]_{\coker v}& C_0 \ar [d]^{\overline{v}}\\
B_0 \ar [r]_{\overline{g}}& E
}
$$
has $E\in\fct{\perp_\infty}{\cb}$.
\end{prop}
\begin{proof}
The first square above is cellular if and only if the unique morphism $t:P\to D$ from the pushout
$$
\xymatrix@C=4pc@R=3pc{
A \ar [r]^{u} \ar [d]_{f}& C \ar [d]^{f'}\\
B \ar [r]_{u'}& P
}
$$
is in $\cm$. It suffices to show that 
$$
\overline{v}\coker g=\coker t.
$$
We have 
$$
\overline{v}\coker g\cdot t\cdot f'= \overline{v}\coker g\cdot g=0
$$
and
$$
\overline{v}\coker g\cdot t\cdot u'= \overline{v}\coker g\cdot v
=\overline{g}\coker v\cdot v=0.
$$
Hence $\overline{v}\coker g\cdot t=0$.

Conversely, let $ht=0$ where $h:D\to X$. Then $hg=htf'=0$ and $hv=htu'=0$.
Thus there are unique $h':C_0\to X$ and $h'':B_0\to X$ such that $$
h'\coker g=h''\coker v=h.
$$
Thus there is a unique $p:E\to X$ such that $p\overline{v}=h'$ and
$p\overline g=h''$. Hence 
$$
p\overline{v}\coker g=h'\coker g=h.
$$
Since $\overline{v}\coker g$ is an epimorphism ($\overline{v}$ is an epimorphism because pushouts preserve epimorphisms), 
$\overline{v}\coker g=\coker t$.
\end{proof}
\appendix

\section{Canonicity of stable independence}

We prove here canonicity of stable independence without the hypothesis, present in \cite[9.1]{indep-categ-advances}, that all morphisms are monomorphisms. This does not depend on the rest of the paper. Our proof is a category-theoretic version of the argument in \cite{bgkv-apal} which shows somewhat more transparently what is going on there. The key notion is that of an independent sequence:

\begin{defin}
  Let $\ck$ be a category and let $\nf$ be an independence notion on $\ck$. Let $f : M_0 \to M$ be a morphism in $\ck$. An \emph{$\nf$-independent sequence for $f$} consists of a nonzero ordinal $\alpha$ and morphisms $(f_i)_{i \le \alpha}$ and $(g_{i, j})_{i \le j \le \alpha}$ such that for $i \le j \le k \le \alpha$:

  \begin{itemize}
  \item $f = f_0$ and $N_0=M$.
  \item $f_i : M \to N_i$ for $0<i$.
  \item $g_{i, j} : N_i \to N_j$.
  \item $g_{j,k} g_{i, j} = g_{i, k}$, $g_{i, i} = \id_{N_i}$.
  \item When $i < j$, the following square commutes and, when $j < \alpha$, is $\nf$-independent:

    $$  
    \xymatrix{
      M \ar[r]^{f_j} & N_j \\
      M_0 \ar[u]_{f_0} \ar[r]_{g_{0, i} f_0} & N_i \ar[u]_{g_{i, j}} \\
    }
    $$
  \end{itemize}

  We call $\alpha$ the \emph{length} of the sequence. For a regular cardinal $\lambda$, we say the independent sequence is \emph{$\lambda$-smooth} if whenever $\cf{i} \ge \lambda$, $N_i$ is the colimit of the system $(g_{j, k})_{j \le k < i}$. We say it is smooth if it is $\aleph_0$-smooth.
\end{defin}

For example, an independent sequence of length one for $f: M_0 \to M$ consists of $f_0 = f$, $f_1 : M \to N_1$, $g_{0, 1} : M = N_0 \to N_1$ such that $f_1 f_0 = g_{0, 1} f_0$. Since there are no independence requirements, it is essentially just the morphism $f_0$ (the additional data is only relevant when $\alpha$ is limit; we could have taken $N_1 = N_0 = M, f_1 = \id_{M}$). More interestingly, an independent sequence of length two consists essentially (because $N_0 = M$ and $g_{0,0} f_0 = f_0$) of an independent square:

    $$  
    \xymatrix{
      M \ar[r]^{f_1} & N_1 \\
      M_0 \ar[u]_{f_0} \ar[r]_{f_0} & M \ar[u]_{g_{0, 1}} \\
    }
    $$

    Thus it consists of two ``independent copies'' of $M$.

    An independent sequence of length three will look like:

    $$  
    \xymatrix{
      & N_2 & & \\
      & & N_1 \ar[ul]_{g_{1, 2}} & \\
      M  \ar[uur]_{f_2} & M \ar[ur]_{f_1} & & M \ar[ul]_{g_{0, 1}} \\
      & & M_0 \ar[ull]_{f_0} \ar[ul]_{f_0} \ar[ur]_{f_0} &
    }
    $$

    where the inner diamond $(M_0, M, M, N_1)$ and the outer diamond $(M_0, M, N_1, N_2)$ is independent (in fact, if $\nf$ is monotonic, all commutative subsquares of the diagram will be independent). Essentially, the leftmost ``copy'' of $M$ is independent of the two rightmost copies (in fact it is independent of $N_1$).

Existence allows us to build independent sequences. Recall that a category $\ck$ has \emph{chain bounds} if any chain has a compatible cocone. 
    
\begin{lem}\label{exist-indep}
  If $\ck$ has $\lambda$-directed colimits, chain bounds, and $\nf$ is a monotonic independence notion with existence, then for any morphism $f: M_0 \to M$ and any ordinal $\alpha$, there exists a $\lambda$-smooth independent sequence for $f$ of length $\alpha$. More generally, any independent sequence of length $\alpha_0 < \alpha$ extends to one of length $\alpha$ (in the natural sense).
\end{lem}
\begin{proof}
  By repeated use of existence.
\end{proof}

The following local character lemma will be handy:

\begin{lem}\label{loc-character}
  Let $\ck$ be a category, $\nf$ an independence relation such that $\ck_{\smallnf}$ is a $\lambda$-accessible category. Let $(M_i \to N_i)_{i < \lambda^+}$ be a system of $\lambda^+$-presentable objects in $\ck^2$ with colimit $M \to N$. Then there exists $i < \lambda^+$ such that the square

  $$  
  \xymatrix{
    N_i \ar[r] & N \\
    M_i \ar[u] \ar[r] & M \ar[u] \\
  }
  $$

  is independent.
\end{lem}
\begin{proof}
  Write $I$ for $\lambda^+$ with the usual ordering. By taking colimits at ordinals of cofinality $\lambda$ and adding them to the system, we can assume without loss of generality that the system is $\lambda$-smooth: for any $i \in I$ of cofinality $\lambda$, $M_i$ is the colimit of $(M_{i_0})_{i_0 < i}$.

  Let $(M_j' \to N_j')_{j \in J}$ be a $\lambda^+$-directed system of $\lambda^+$-presentable objects whose colimit in $\ck_{\smallnf}$ is $M \to N$; we know that $\ck_{\smallnf}$ is $\lambda^+$-accessible. We build $(i_\alpha, j_\alpha)_{\alpha < \lambda}$ such that for all $\alpha < \lambda$:

  \begin{enumerate}
  \item $i_\alpha \in I$, $j_\alpha \in J$.
  \item $i_\alpha < i_{\alpha + 1}$.
  \item The map from $M_{i_\alpha} \to N_{i_\alpha}$ to $M \to N$ factors through $M_{j_\alpha}' \to N_{j_\alpha}'$.
  \item The map from $M_{j_\alpha}' \to N_{j_\alpha}'$ to $M \to N$ factors through $M_{i_{\alpha + 1}} \to N_{i_{\alpha + 1}}$.
  \end{enumerate}

  This is possible since $I$ and $J$ are $\lambda^+$-directed and $M_i \to N_i$, $M_j' \to N_j'$ are $\lambda^+$-presentable. Now, let 
  $i := \sup_{\alpha < \lambda} i_\alpha$. The colimit in $\ck^2$ of $(M_{i_\alpha} \to N_{i_\alpha})_{\alpha < \lambda}$ and 
  $(M_{j_\alpha}' \to N_{j_\alpha}')_{\alpha < \lambda}$ coincide and by $\lambda$-smoothness must be $M_i \to N_i$. By assumption, for all $\alpha < \lambda$, the square

  $$  
  \xymatrix{
    N_{j_\alpha}' \ar[r] & N \\
    M_{j_\alpha}' \ar[u] \ar[r] & M \ar[u] \\
  }
  $$

  is independent. Since $\ck_{\smallnf}$ has $\lambda$-directed colimits, this means that the square

    $$  
  \xymatrix{
    N_i \ar[r] & N \\
    M_i \ar[u] \ar[r] & M \ar[u] \\
  }
  $$

  is also independent.
\end{proof}

A much simpler result than the canonicity theorem is:

\begin{lem}\label{one-dir-lem}
  Assume $\ck$ is a category, $\nf^1$, $\nf^2$ are independence notions such that $\nf^1 \subseteq \nf^2$, $\nf^1$ has existence, and $\nf^2$ has uniqueness. Then $\nf^1 = \nf^2$.
\end{lem}
\begin{proof}
  Given a square $M_0, M_1, M_2, M_3$ that is $\nf^2$-independent, use existence for $\nf^1$ to $\nf^1$-amalgamate the span $M_0 \to M_1$, $M_0 \to M_2$, giving maps $M_1 \to M_3'$, $M_2 \to M_3'$. Now by uniqueness for $\nf^2$, the amalgam involving $M_3$ and the one involving $M_3'$ must be equivalent, hence $M_0, M_1, M_2, M_3$ is also $\nf^1$-independent. 
\end{proof}

We can now prove the canonicity theorem. The idea is to use a generalization of the fact that, in a vector space, if $I$ is linearly independent and $a$ is a vector, there exists a finite subset $I_0 \subseteq I$ such that $(I - I_0) \cup \{a\}$ is independent. Thus we can remove a small subset of $I$ and get something independent. 

\begin{lem}\label{canon-key-lem}
  Assume $\ck$ has chain bounds, and $\nf^1$, $\nf^2$ are independence notions with existence such that:

  \begin{enumerate}
  \item $\nf^1$ is right monotonic.
  \item $\nf^2$ is transitive, left monotonic, and right accessible.
  \end{enumerate}

  Then any span has an amalgam that is both $\nf^1$-independent and $\nf^2$-independent. In particular, if $\nf^1$ has uniqueness then $\nf^1 \subseteq \nf^2$.
\end{lem}
\begin{proof}
  Consider a span $M_0 \xrightarrow[f_0]{} M$, $M_0 \xrightarrow[f_0']{} M'$. Fix a regular cardinal $\lambda$ such that $\ck_{\smallnf^2}$ (the arrow category induced by $\nf^2$) is $\lambda$-accessible and $M_0, M, M', f_0, f_0'$ are $\lambda$-presentable in all relevant categories. 

  Using Lemma \ref{exist-indep}, build a $(\nf^2)^d$-independent sequence for $f_0$, $(f_i : M \to N_i)_{i \le \lambda^+}$, $(g_{i, j}: N_i \to N_j)_{i \le j \le \lambda^+}$, where $N_{\lambda^+}$ is the colimit of $(N_i)_{i < \lambda^+}$. Observe that
  $$
  f_{\lambda^+}f_0=g_{0,\lambda^+}f_0.
  $$
  Along the way, we ensure that $N_i$ is $\lambda^+$-presentable for $i < \lambda^+$. Now $\nf^1$-amalgamate the span 
  $M_0 \to N_{\lambda^+}$, $M_0 \to M'$, giving an $\nf^1$-independent square:

  $$
  \xymatrix{
    M' \ar[r]^{h'} & N_{\lambda^+}' \\
    M_0 \ar[u]^{f_0'} \ar[r]_{g_{0,\lambda^+}f_0} & N_{\lambda^+} \ar[u]_{h} \\  
  }
  $$

  with $N_{\lambda^+}'$ a $\lambda^{++}$-presentable object. Reworking the proof of \cite[Lemma 1]{rosicky-sat-jsl}---which requires directed colimits---to use the chain bounds available to us here, we can write $N_{\lambda^+}'$ as a colimit of $\lambda^+$-presentables $(g_{i,j}':N_i'\to N_{j}')_{i\leq j < \lambda^+}$, where:
  \begin{enumerate}\item There is an arrow $h_i:N_i \to N_i'$ for each $i < \lambda^+$.
  \item The $N_i'$ lie above $M'$, in the sense that $h':M'\to N_{\lambda^+}'$ factors as
  $$M'\stackrel{u_i}{\longrightarrow}N_i'\stackrel{g_{i,\lambda^+}'}{\longrightarrow}N_{\lambda^+}'$$ and, moreover, that the morphisms $h'f_0'=hg_{0,\lambda^+}f_0:M_0\to N_{\lambda^+}'$ factor identically through $g_{i,\lambda^+}'$, i.e.
  $$h_if_if_0=u_if_0'.$$
  Here we use $\lambda$-presentability of $M_0$, $M'$, and $\lambda^+$-directedness of the chain.
  \end{enumerate}
 Then $h$ is a colimit of the $h_i$ in $\ck^2$ and by Lemma \ref{loc-character}, there exists $i < \lambda^+$ such that the square

  $$
  \xymatrix{
    N_i' \ar[r]^{g_{i,\lambda^+}'} & N_{\lambda^+}' \\
    N_i \ar[u]^{h_i} \ar[r]_{g_{i,\lambda^+}} & N_{\lambda^+} \ar[u]_{h} \\  
  }
  $$
  
    is $\nf^2$-independent. By definition of an $(\nf^2)^d$-independent sequence, the square

  $$  
  \xymatrix{
    N_i \ar[r]^{g_{i,\lambda^+}} & N_{\lambda^+} \\
    M_0 \ar[u]^{f_if_0} \ar[r]_{f_0} & M \ar[u]_{f_{\lambda^+}} \\
  }
  $$

  is $\nf^2$-independent. By left transitivity, we obtain that the following is $\nf^2$-independent.

  $$  
  \xymatrix{
    N_i' \ar[r]^{g_{i,\lambda^+}'} & N_{\lambda^+}' \\
    M_0 \ar[u]^{h_if_if_0} \ar[r]_{f_0} & M \ar[u]_{hf_{\lambda^+}} \\
  }
  $$

  A chase through the diagrams above reveals that 
  $$
  g_{i,\lambda^+}'h_if_if_0=h'f_0'=hg_{0,\lambda^+}f_0=hf_{\lambda^+}f_0,
  $$ 
  meaning that the outer square and the large upper triangle in the following diagram commute:
  
    $$  
  \xymatrix{
    N_i' \ar[rr]^{g_{i,\lambda^+}'} &  & N_{\lambda^+}' \\
     & M'\ar[ur]_{h'}\ar[ul]^{u_i} & \\
    M_0 \ar[uu]^{h_if_if_0}\ar[ur]_{f_0'} \ar[rr]_{f_0} & & M \ar[uu]_{hf_{\lambda^+}} \\
  }
  $$
  
 Thus the square
 $$  
  \xymatrix{
    N_i' \ar[r]^{g_{i,\lambda^+}'} & N_{\lambda^+}' \\
    M_0 \ar[u]^{u_if'_0} \ar[r]_{f_0} & M \ar[u]_{hf_{\lambda^+}} \\
  }
  $$
is $\nf^2$-independent.
    
  By left monotonicity for $\nf^2$, then, the following is also $\nf^2$-independent:  

    $$  
  \xymatrix{
    M' \ar[r]^{h'} & N_{\lambda^+}' \\
    M_0 \ar[u]^{f_0'} \ar[r]_{f_0} & M \ar[u]_{hf_{\lambda^+}} \\
  }
  $$

  Note, however, that the morphism from $M$ to $N_{\lambda^+}'$ in the diagram above is not the same as the one in the $\nf^1$-amalgam of $M_0 \to N_{\lambda^+}$, $M_0 \to M'$. In fact, we have a diagram of the form:

      $$  
  \xymatrix{
    M' \ar[rr]^{h'} & & N_{\lambda^+}' \\
    M_0 \ar[u]^{f_0'} \ar[r]^{f_0} \ar[rd]_{f_0} & M \ar[r]^{g_0,\lambda^+} & N_{\lambda^+} \ar[u]_{h} \\
    & M \ar[ru]_{f_{\lambda^+}} & \\
  }
  $$
  
  where the upper rectangle is $\nf^1$-independent and the outer ``square'' $(f_0', f_0, h f_{\lambda^+}, h')$ is $\nf^2$-independent. By right monotonicity for $\nf^1$, we get that $(f_0', f_0, h f_{\lambda^+}, h')$ is also $\nf^1$-independent. Thus it is the desired amalgam of $f_0', f_0$. 
\end{proof}

\begin{thm}[The canonicity theorem]\label{canon-thm}
  Assume $\ck$ has chain bounds, and $\nf^1$, $\nf^2$ are independence notions with existence and uniqueness such that:

  \begin{enumerate}
  \item $\nf^1$ is right monotonic.
  \item $\nf^2$ is transitive and right accessible.
  \end{enumerate}

  Then $\nf^1 = \nf^2$. In particular, $\ck$ has at most one stable independence notion.
\end{thm}
\begin{proof}
  Combine Lemmas \ref{one-dir-lem} and \ref{canon-key-lem}. Note that right monotonicity for $\nf^2$ follows from existence, uniqueness, and transitivity \cite[3.20]{indep-categ-advances}.
\end{proof}

\begin{cor}
  Assume $\ck$ has chain bounds. If $\nf$ is a transitive and right accessible independence notion with existence and uniqueness, then $\nf$ is a stable independence notion. In particular, it is symmetric.
\end{cor}
\begin{proof}
  It suffices to see that $\nf = \nf^d$. For this, apply Theorem \ref{canon-thm} with $\nf^1 = \nf$ and $\nf^2 = \nf^d$ (again, $\nf^2$ is right monotonic by \cite[3.20]{indep-categ-advances}).
\end{proof}
\begin{remark}
  Instead of chain bounds, it suffices to be able to build the appropriate independent sequences. See \cite[9.6]{indep-categ-advances}.
\end{remark}

\bibliographystyle{amsalpha}
\bibliography{more-indep}

\end{document}